\documentclass[a4paper,12pt]{article}

\usepackage{amsmath}
\usepackage{amssymb}
\usepackage{amsthm}
\usepackage{graphicx}
\usepackage{amscd}
\usepackage{epic, eepic}
\usepackage{url}
\usepackage{color}
\usepackage[utf8]{inputenc} 
\usepackage{comment}
\usepackage{enumerate}   
\usepackage{graphicx}
\usepackage{epstopdf}
\usepackage{enumitem}
\usepackage{tikz}
\usepackage[top=24mm, bottom=25mm, left=25mm, right=25mm]{geometry}
\usepackage[bookmarks=false,hidelinks]{hyperref}
\usepackage{pgfplots}
\pgfplotsset{compat=1.15}
\usepackage{mathrsfs}
\usetikzlibrary{arrows}
\definecolor{wrwrwr}{rgb}{0.3803921568627451,0.3803921568627451,0.3803921568627451}
\usepackage{mathtools}

\usepackage{bbm} 

\newcommand{\floor}[1]{\left\lfloor{#1}\right\rfloor}
\newcommand{\ceil}[1]{\left\lceil{#1}\right\rceil}

\newtheorem{theorem}{Theorem}
\newtheorem{claim}{Claim}

\newtheorem{conjecture}{Conjecture}

\newtheorem{observation*}{Observation}
\newtheorem{lemma}{Lemma}

\newcommand{\ex}{{\rm  ex}}

\usepackage{authblk}

\newcommand{\Mod}[1]{\ (\mathrm{mod}\ #1)}

\title{Generalized Tur\'an numbers for the edge blow-up of a graph}
\author[1,3]{Zequn Lv}
\author[1]{Ervin Győri}
\author[1,3]{Zhen He\footnote{Correspondoing author. Email: {hz18@mails.tsinghua.edu.cn}}}
\author[1,5]{Nika Salia}
\author[1]{Casey Tompkins}
\author[1,4]{Kitti Varga}
\author[1,2]{Xiutao Zhu}
\date{}

\affil[1]{Alfr\'ed R\'enyi Institute of Mathematics, Hungarian Academy of Sciences. }
\affil[2]{Department of Mathematics, Nanjing University.}
\affil[3]{Department of Mathematical Sciences, Tsinghua University.}
\affil[4]{Department of Computer Science and Information Theory, Budapest University of Technology and Economics.}
\affil[5]{Extremal Combinatorics and Probability Group, Institute for Basic Science, Daejeon, South Korea.}
\begin{document}
\maketitle
\begin{abstract}
Let $H$ be a graph and $p$ be an integer.
The edge blow-up $H^p$ of $H$ is the graph obtained from replacing each edge in $H$ by a copy of $K_p$ where the new vertices of the cliques are all distinct. 
Let $C_k$ and $P_k$ denote the cycle and path of length~$k$, respectively.
In this paper, we find sharp upper bounds for $\ex(n,K_3,C_3^3)$ and the exact value for $\ex(n,K_3,P_3^3)$. Moreover, we determine the graphs attaining these bounds.
\end{abstract}

\section{Introduction}
 \textbf{Notation. }In this paper, we use $C_k$, $P_k$, $M_k$ and $S_k$ to denote the cycle, path, matching and star with $k$ edges, respectively. 
 Let $K_t$ be the complete graph on $t$ vertices and  $K_{s,t}$ be the complete bipartite graph with parts of size $s$ and $t$. 
 The vertex and edge sets of a graph $G$ are denoted by $V(G)$ and $E(G)$, respectively.  
 Also we denote the number of edges in $G$ by $e(G)$.
 For two graph $G$ and $H$, let $G\cup H$ denote the disjoint union of $G$ and $H$. 
 Let $G+H$ denote the join of $G$ and $H$, which is obtained from $G\cup H$ by adding all edges with one endvertex in $V(G)$ and the other endvertex in $V(H)$.  
 Let $T(G)$ denote the set of all triangles in $G$ and $t(G)=|T(G)|$.
 For a vertex $v$ in $V(G)$, let $t(v)$ denote the number of triangles containing $v$. 
 For an edge~$uv$, let $N(uv)=N(u)\cap N(v)$. 
Hence, $|N(uv)|$ is the number of triangles containing $uv$. 
For a set of vertices $S \subseteq V(G)$ we denote by $G[S]$ the induced subgraph of $G$ on $S$ and we set $G-S=G[V(G) - S]$.
For two disjoint sets of vertices $U,W \subseteq V(G)$ we denote by $G[U,W]$ the bipartite subgraph of $G$ consisting of those edges with one endvertex in $U$ and the other in $W$. 

\vskip 3mm
Let $H$ be a given graph and $p$ be an integer greater than $2$. 
The edge blow-up $H^p$ of $H$ is the graph obtained from replacing each edge in $H$ by a copy of $K_p$ where the new vertices of the cliques are all distinct.
The problem of finding the Tur\'an number of $H^p$ for various graphs $H$ has attracted a lot of attention.
The first results on the topic can be dated back to 1960s.  
Moon~\cite{Moon}, and independently  Simonovits~\cite{simonovits1968method} determined the Tur\'an number $\ex(n,M_k^p)$ for $p\ge 3$. 
Much later Erd\H{o}s, F\"uredi, Gould and Gunderson~\cite{ERDOS199589} determined the Tur\'an number $\ex(n,S_k^p)$ for $p=3$, and then Chen, Gould, Pfender and Wei~\cite{chen2003extremal} extended this result to any $p\ge 3$.
Glebov~\cite{glebov2011extremal} determined the Tur\'an number of~$P_k^p$. 
More recently, Liu extended Glebov's result to the edge blow-up of a family of trees and also determined $C_k^p$ for sufficiently large $n$.  
Wang, Hou, Liu and Ma~\cite{Wang2021TheTN} determined the $\ex(n,T^P)$ for a larger family of trees and Yuan~\cite{Yuan2022ExtremalGF} determined $\ex(n,H^P)$ for any non-bipartite graph $H$ and $p\ge \chi(H)+1$.   

We will make use of the following result of Xiao, Katona, Xiao and Zamora~\cite{XIAO20221}, which determined  the value of $\ex(n,C_3^3)$ for all $n\ge 6$.

\begin{theorem}(Xiao, Katona, Xiao and Zamora~\cite{XIAO20221})\label{edge verion}
Let $n\ge 6$ be an integer, then
\begin{align*}
\ex(n, C_3^3)=
\begin{cases} 
\lfloor\frac{n^2}{4}\rfloor +\lfloor\frac{n}{2}\rfloor  & \mbox{if  $n \not\equiv 2\Mod 4$},\\
\frac{n^2}{4}+\frac{n}{2}-1 & \mbox{if\/ $n \equiv 2\Mod 4$}.\\
\end{cases}
\end{align*}
When $n=4k$, $M_{\frac{n}{4}}+M_{\frac{n}{4}}$ is the unique extremal graph.\\
When $n=4k+1$, $(M_{\lfloor\frac{n}{4}\rfloor}\cup K_1)+M_{\frac{n-1}{4}}$ and $S_{\lfloor\frac{n}{2}\rfloor}+\overline{K_{\lfloor\frac{n}{2}\rfloor}}$ are the extremal graphs.\\
When $n=4k+2$,  $(M_{\lfloor\frac{n}{4}\rfloor}\cup K_1)+(M_{\lfloor\frac{n}{4}\rfloor}\cup K_1)$, $M_{\lceil\frac{n}{4}\rceil}+M_{\lfloor\frac{n}{4}\rfloor}$ and $S_{\frac{n}{2}-1}+\overline{K_\frac{n}{2}}$ are the extremal graphs.\\
When $n=4k+3$, $(M_{\lfloor\frac{n}{4}\rfloor}\cup K_1)+M_{\lceil\frac{n}{4}\rceil}$ and $S_{\lfloor\frac{n}{2}\rfloor}+\overline{K_{\lfloor\frac{n}{2}\rfloor}}$ are the extremal graphs.
\end{theorem}

In this paper, we will consider the generalized Tur\'an number.
Let $T$ and $H$ be graphs, then the generalized Tur\'an number $\ex(n,T,H)$ is the maximum number of copies of $T$ that an $n$-vertex $H$-free graph $G$ can contain. 
If $T=K_2$, then $\ex(n,T,H)$ is the classical Tur\'an number of $H$.

Although several results about the Tur\'an number of an edge blow-up of a graph have been obtained, less is known about the generalized Tur\'an number of such graphs.  
However, there have been some results in this direction.  
Liu and Wang~\cite{liu2021generalized} determined the value of $\ex(n,K_p, S_2^p)$ and $\ex(n,K_p, M_2^p)$. 
Later Gerbner and  Patk\'os~\cite{gerbner2021generalized} determined $\ex(n,K_r, S_2^p)$ and $\ex(n,K_r, M_2^p)$ for any $r$, $p$, and Yuan and Yang~\cite{Yuanxiaoli}  determined $\ex(n,K_3, M_2^3)$ for all~$n$. 
Recently, Zhu, Chen, Gerbner, Gy\H{o}ri and Hama Kairm~\cite{zhufriendship} determined $\ex(n,K_3, S_k^3)$ for any $k$.

Our results concern the edge blow-ups of cycles and paths. 
We prove the following theorems.

\begin{theorem}\label{Cycle}
Let $n\ge 22$ be an integer, we have 
\[\ex(n,K_3, C_3^3)\le \frac{n^2}{4}-1+\mathbbm{1}_{4|n},\]
where $\mathbbm{1}_{4|n}$ is the indicator function for $4|n$.
Furthermore, equality holds when $n$ is even and   $M_{\lceil\frac{n}{4}\rceil}+M_{\lfloor\frac{n}{4}\rfloor}$ is the unique extremal graph.
\end{theorem}

\begin{theorem}\label{Path}
Let $n\ge 300^3$ be an integer. We have
\[\ex(n,K_3, P_3^3)=\left\lfloor\frac{(n-1)^2}{4}\right\rfloor,\]
and the unique extremal graph is  $K_1+K_{\floor{\frac{n-1}{2}},\ceil{\frac{n-1}{2}}}$.
\end{theorem}

The rest of the paper is organized as follows. In Section~\ref{ProofC3}, we prove Theorem~\ref{Cycle}. In Section~\ref{ProofP}, we prove Theorem~\ref{Path}. In Section~\ref{concl}, we mention some problems about the general case: $\ex(n,K_3, C_k^3)$ and $\ex(n,K_3,P_k^3)$.

\section{Proof of Theorem \ref{Cycle}}\label{ProofC3}
One can see that when $n$ is even, the graph $M_{\lceil\frac{n}{4}\rceil}+M_{\lfloor\frac{n}{4}\rfloor}$  contains $\frac{n^2}{4}-1+\mathbbm{1}_{4|n}$ triangles. 
So our aim is to show that $\ex(n,K_3,C_3^3)\le \frac{n^2}{4}-1+\mathbbm{1}_{4|n}$. 

Let $n\ge 22$ be an integer and $G$ be an $n$-vertex $C_3^3$-free graph with $t(G)$ being the maximum.
We may assume every edge is contained in at least one triangle, otherwise we delete this edge.

We define the weight of $uv$ by 
\[
w(uv)\coloneqq \frac{1}{|N(uv)|}.
\]
For a triangle $xyz$, its weight is defined by $w(xyz)=w(xy)+w(xz)+w(yz)$.
We will prove the upper bound by making use of the following claims.
\begin{claim}\label{ weight of triangle}
For any triangle $xyz$ in $G$, 
$$1+\frac{1}{n-2}\le w(xyz)\le 3,$$
or $w(xy)=w(yz)=w(xz)=\frac{1}{3}$ and there exists another two vertices $u,v$ such that $\{x,y,z,u,v\}$ induces a copy of $K_5^-$ or $K_5$.
\end{claim}

\begin{proof}
Since each edge is contained in at least one triangle, without loss of generality, we have
\[
\frac{1}{n-2}\le w(yz)\le w(xz)\le w(xy)\le 1.
\]
If $w(xy)=1$, then  $w(xyz)\ge 1+\frac{2}{n-2}$ and we are done. 
Next we may assume $w(xy)\le \frac{1}{2}$ and we distinguish two cases based on whether $w(xy)=\frac{1}{2}$ or $w(xy)\le \frac{1}{3}$ .
 
First suppose $w(xy)=\frac{1}{2}$ and let $N(xy)=\{z,z'\}$.
If $w(xz)=\frac{1}{2}$, then $w(xyz)\ge 1+\frac{1}{n-2}$ and we are done. Thus we may assume $w(xz)\le \frac{1}{3}$ and let $y'\in N(xz)-\{y,z'\}$. If $w(yz)\le \frac{1}{4}$, then we can find a vertex $x'\in N(yz)-\{x,y',z'\}$ and $\{x,y,z,x',y',z'\}$ contains a copy of $C_3^3$, a contradiction. 
Hence $w(yz)=w(xz)=\frac{1}{3}$ and $w(xyz)=\frac{7}{6}\ge 1+\frac{1}{n-2}$, inequality holds since $n\ge 22$. 

Now suppose $w(xy)\le \frac{1}{3}$.  
Let $u,v \in N(xy)-\{z\}$. 
If $w(yz)\le \frac{1}{4}$, then there is a vertex $x'\in N(yz)-\{u,v,x\}$. 
Also we can find a vertex $y'\in N(xz)-\{y,x'\}$ and another vertex in $\{u,v\}$ not equal to $y'$ (say $u\not=y'$).
Then $\{y',x',u,x,y,z\}$ contains a copy of $C_3^3$, a contradiction. 
It follows that $w(xz)=w(yz)=w(xy)=\frac{1}{3}$.
Furthermore, if $N(yz)-\{x\}$ or $N(xz)-\{y\}$ is not equal to $\{u,v\}$, then one can check that we still can find a copy of $C_3^3$, a contradiction. Hence $\{x,y,z,u,v\}$ induces a copy of $K_5^-$ or $K_5$.
\end{proof}

\begin{claim}\label{t<e}
$t(G)\le e(G).$
\end{claim}
\begin{proof} By Claim~\ref{ weight of triangle}, we have
 \[t(G)=\sum_{xyz\in T(G)}1\le \sum_{xyz\in T(G)}\left(w(xz)+w(yz)+w(xy)\right)= e(G),\]
 as required.
 \end{proof}

\begin{claim}\label{no K_5}
For any triangle $xyz$, $w(xyz)\ge 1+\frac{1}{n-2}$, i.e., there is no $K_5^-$ in $G$.
\end{claim}
\begin{proof}
Suppose to the contrary that there is a subgraph $H$ of $G$ isomorphic to $K_5$ or $K_5^-$ induced on the set $\{v_1,v_2,v_3,v_4,v_5\}$.
If $H$ is isomorphic to $K_5^-$, then we may assume without loss of generality $v_4v_5$ is not an edge.

One can check that for any edge $v_iv_j$ in $H$, $N(v_iv_j)\subseteq V(H)$. 
Otherwise we can find a copy of $C_3^3$.
Let $S=(N(v_4)\cap N(v_5))-V(H)$ if $H$ is isomorphic to $K_5^-$ and $S=\emptyset$ otherwise.  

If $|S|\le  n-10$, then $e(G-V(H))\le \frac{(n-5)^2}{4}+\frac{n-5}{2}$ by Theorem~\ref{edge verion}, and we have
\begin{align*}
e(G)\le& ~e(H)+e(G[V(H), V(G)\setminus V(H)])+e(G-V(H))\\
\le&~10+(n-5)+|S|+ \frac{(n-5)^2}{4}+\frac{n-5}{2}\\
< & ~\frac{n^2}{4}-1.
\end{align*}
By Claim~\ref{t<e}, it follows that $G$ is not the extremal graph.

If $|S|> n-10$, then $G[S]$ is $P_3$-free, otherwise together 
with $v_4,v_5$, we can find a copy of $C_3^3$.
Hence $e(G-V(H))\le (n-5-|S|)|S|+|S|+\binom{n-5-|S|}{2}$. When $n\ge 22$, we have
\begin{align*}
e(G)
\le~ & 9+(n-5)+|S|+ (n-5-|S|)|S|+|S|+\binom{n-5-|S|}{2}\\
\le~ & 8n-56\\
<~&\frac{n^2}{4}-1.
\end{align*}
Again by Claim \ref{t<e}, we are done.
\end{proof}

Let $T_1(G)=\{xyz\in T(G): w(xyz)\ge 1+\frac{2}{n}\}$ and $T_2(G)=T(G)-T_1(G)$. 
We have the following bound on the average weight of a triangle in $G$.
\begin{claim}\label{average weight}
The average weight of each triangle in $G$ is at least $1+\frac{2}{n}$.    
\end{claim}

\begin{proof}
 If $T_2(G)$ is empty, then there is nothing to prove.  Hence we may assume $T_2(G)\not=\emptyset$.  

Let $xyz$ be a triangle in $T_2(G)$ with $w(xy)\ge w(xz)\ge w(yz)$. 
By Claim~\ref{ weight of triangle} and~\ref{no K_5}, we have $w(xy)\in \{1, \frac{1}{2}\}$. If $w(xy)=1$, then 
$$w(xyz)\ge 1+\frac{2}{n-2}$$ which means $xyz\in T_1(G)$, a contradiction. So $w(xy)=\frac{1}{2}$. Let $N(xy)=\{z,~x'\}$. 
Suppose $N(xz)-\{y,x'\}\not=\emptyset$ and $y'\in N(xz)-\{y,x'\}$.
Then either $N(yz)-\{x,~x',~y'\}\not=\emptyset$ and we can find a copy of $C_3^3$, or $N(yz)=\{x,~x',~z'\}$ which means $\{x,y,z,x',y'\}$ contains a copy of $K_5^-$, or 
$w(yz)\ge \frac{1}{2}$ which means $w(xyz)=\frac{3}{2}\ge 1+\frac{2}{n}$. In all of these cases, we get a contradiction. Hence, $N(xz)=\{y,~x'\}$ and $w(xy)=w(xz)=\frac{1}{2}$ and so $w(yz)\le \frac{2}{n}$.
For the edges $x'y, ~x'z$, we  deduce that $w(x'y)=w(x'z)=\frac{1}{2}$.  If not, suppose $w(x'y)<\frac{1}{2}$. Let $u$ be in $N(x'y)-\{x,z\}$. Since $|N(yz)|>\frac{n}{2}$, let $v$ be in $N(yz)-\{x,x',u\}$. Then  $\{u,v, x',x,y,z\}$ contains a copy of $C_3^3$, a contradiction. It follows that the triangle $x'yz$ is also in $T_2(G)$.

Therefore, for any triangle $xyz$ in $T_2(G)$, there is a unique triangle $x'yz$ in $T_2(G)$ such that  $\{x,x',y,z\}$ induces a copy of $K_4$ and $w(xy)=w(xz)=w(x'y)=w(x'z)=\frac{1}{2}$.
Hence, we can partition the set of triangles in $T_2(G)$ into pairs $(xyz,~x'yz)$. For each such pair, we define a mapping
\[\phi (xyz,~x'yz)=\{xyz,~x'yz,~xx'y,~xx'z\}.\]
Note that since $w(x'z)=\frac{1}{2}$ and $N(x'z)=\{x,y\}$, then $N(xx')\cap N(yz)=\emptyset$ and $|N(xx')|< n-|N(yz)|\le \frac{n}{2}$. 
This means $xx'y, xx'z$  are in  $T_1(G)$. Furthermore, by Claim \ref{no K_5}, each triangle is contained in at most one copy of $K_4$, so $xx'y, xx'z$ do not belong to any other $\phi (uvw,~u'vw)$. Since
\begin{align*}
&w(xyz)+w(x'yz)+w(xx'y)+w(xx'z)\\
&=4+\frac{2}{|N(yz)|}+\frac{2}{N(xx')}\\
& \ge 4+\frac{8}{|N(xx')|+|N(yz)|}\ge 4+\frac{8}{n},
\end{align*}
we can transfer the weight of $xx'y, ~xx'z$ to $xyz,~x'yz$ and ensure  the average weight is at least $1+\frac{2}{n}$.
\end{proof}

Now by the Claim \ref{average weight} and  Theorem \ref{edge verion}, we have 
\begin{align*}
t(G)&=\sum_{xyz\in T(G)}1\le \frac{n}{n+2}\sum_{xyz\in T(G)}1+\frac{2}{n}\\
&\le \frac{n}{n+2}\sum_{xyz\in T(G)}\left(w(xz)+w(yz)+w(xy)\right)\\
&\le \frac{n}{n+2}e(G)\le \frac{n^2}{4}-1+\mathbbm{1}_{4|n}.
\end{align*}
Equality holds if and only if $e(G)$ attains the maximum and the average weight of each triangle is exactly $1+\frac{2}{n}$.
Hence, by the characterization of the extremal graphs for $\ex(n,C_3^3)$ in Theorem~\ref{edge verion}, we have  $G=M_{\lceil\frac{n}{4}\rceil}+M_{\lfloor\frac{n}{4}\rfloor}$ when $n$ is even. $\hfill\blacksquare$

\section{Proof of Theorem \ref{Path}}\label{ProofP}
 Let $t(u,v)$ denote the number of triangles containing $u$ or $v$ or both. First, we use a technique to reduce the proof of Theorem \ref{Path} to the case that each vertex is contained in many triangles. To this end we use the following lemma.
\begin{lemma} \label{reduce}
Suppose $G$ is a  $P_3^3$-free graph on at least $300$ vertices. If for any two different vertices $u, v$, we have $t(u), t(v)\ge \frac{n}{2}-1$ and $t(u,v)\ge n-2$, then $t(G)\le \left\lfloor\frac{(n-1)^2}{4}\right\rfloor$ and equality holds if and only if $G=K_1+K_{\floor{\frac{n-1}{2}},\ceil{\frac{n-1}{2}}}$.
\end{lemma}
\noindent First we will show how to deduce Theorem~\ref{Path} from Lemma~\ref{reduce}, then we will prove Lemma~\ref{reduce}.

\subsection{Proof of Theorem~\ref{Path} using Lemma~\ref{reduce}.}
Let $G$ be a $P_3^3$-free graph on $n$ vertices with $n\ge 300^3$ and $t(G)\ge \left\lfloor\frac{(n-1)^2}{4}\right\rfloor$. 
We initialize $G_n=G$ and define a process of as follows: for $j<n$, let $G_j=G_{j+1}-v_1$ if $t(v_1)<\frac{j+1}{2}-1$ in $G_{j+1}$,
or $G_{j-1}=G_{j+1}-\{v_1,v_2\}$ if $t(v_1,v_2)<(j+1)-2$ in $G_{j+1}$. Suppose the process ends at $G_\ell$ and for any two  vertices $u,v$ in $G_\ell$, we have  $t(u),t(v)\ge \frac{\ell}{2}-1$ and $t(u,v)\ge \ell-2$. Note that 
\[\binom{\ell}{3}\ge t(G_\ell)\ge \left\lfloor\frac{(\ell-1)^2}{4}\right\rfloor+\frac{n-\ell}{2}\]
Hence $\ell\ge \sqrt[3]{3n}\ge 300$ and by Lemma 1,  $G_\ell$ contains a copy of $P_3^3$, a contradiction.

That is to say, $G_n$ satisfies the conditions in Lemma 1 and we are done. $\hfill\blacksquare$

\subsection{Proof of Lemma 1.}
Let $G=G_1\cup \cdots \cup G_c$ be a  $P_3^3$-free graph on $n\ge 300$ vertices, where $G_i$ are the connected components of $G$, for $1\le i\le c$.  We may assume each edge of $G$ is contained in at least one triangle, otherwise we delete it and the conditions still hold in the resulting graph.
For any two distinct vertices $u, v$, we have $t(u), t(v)\ge \frac{n}{2}-1$ and $t(u,v)\ge n-2$.
It follows that $v(G_i)\ge \delta(G)\ge \sqrt{n}$.  

As mentioned in the introduction, Yuan and Yang~\cite{Yuanxiaoli} determined $\ex(n,K_3,M_2^3)$ for all $n$.
\begin{theorem}(Yuan and Yang \cite{Yuanxiaoli})\label{2K_3}
For $n\ge 7$, we have
\[\ex(n,K_3,M_2^3)=\max\left\{3n-8, \left\lfloor\frac{(n-1)^2}{4}\right\rfloor\right\}.\]
Furthermore, $K_3+\overline{K}_{n-3}$ or $K_1+K_{\floor{\frac{n-1}{2}},\ceil{\frac{n-1}{2}}}$ is the unique extremal graph.
\end{theorem}

\noindent  If no $G_i$ contains two vertex-disjoint triangles, then since $v(G_i)\ge \sqrt{n}\ge \sqrt{300} $, we have $t(G_i)\le \left\lfloor\frac{(v(G_i)-1)^2}{4}\right\rfloor$
by Theorem~\ref{2K_3} and 
\[t(G)= \sum_{i=1}^c t(G_i)\le  \left\lfloor\frac{(n-1)^2}{4}\right\rfloor.\]
Equality holds if and only if $G$ is connected and $G=K_1+K_{\floor{\frac{n-1}{2}},\ceil{\frac{n-1}{2}}}$.

Therefore, we may assume without loss of generality that $G_1$ contains two vertex-disjoint triangles.
We define the distance between two vertex-disjoint triangles as the minimum length of a path with endvertices in the respective triangles.
Among all vertex-disjoint triangle pairs in $G_1$, let $x_1y_1z_1,~x_2y_2z_2$ be two vertex~disjoint triangles whose distance is the smallest and let $P=x_1\cdots y_2$ be a path of minimal length between them.
First suppose the length of $P$ is at least $2$. 
Let $x_1^+$ be the vertex adjacent to $x_1$ on the path $P$ and let $x_1x_1^+w$ be a triangle containing the edge $x_1x_1^+$.
Then we either find a copy of $P_3^3$ if $w\in \{x_2,~y_2,~z_2\}$, or we find another two vertex disjoint triangles whose distance is smaller, and in both cases we obtain a contradiction. 
Hence we have that $P=x_1y_2$ is a single edge. 

Note that $x_1y_2$ is also contained in a triangle and the third vertex of this triangle must be in $
\{y_1,z_1,x_2,z_2\}$. Without loss of generality, say $x_1y_2z_2$ is a triangle.
Let $S=(N(y_2)\cap N(z_2))-\{x_1,y_1,z_1\}$. Obviously, we have that
$S$ is nonempty  and  independent, since $x_2\in S$ and $G$ contains no copy of $P_3^3$. 

Suppose $u$ is a vertex in $N(y_2)-(S\cup \{x_1,y_1,z_1,z_2\})$ and $uy_2w$ is a triangle containing the edge $uy_2$. If $w$ does not belong to $\{x_1,y_1,z_1\}$, then $y_1z_1x_1,~x_1z_2y_2,~y_2uw$ forms a copy of $P_3^T$. If $w\in \{x_1,y_1,z_1\}$, then $y_1z_1x_1,~wuy_2,~y_2z_2x_2$ forms a copy of $P_3^T$. In both cases we have a contradiction.  It follows that $N(y_2)\subset S\cup \{x_1,y_1,z_1, z_2\}$ and analogously  $N(z_2)\subset S\cup \{x_1,y_1,z_1, y_2\}$.

Since $\delta(G)\ge \sqrt{n}$, then $|S|\ge 2$.
We may assume that for each edge of the form $y_2u$ and $z_2u$ with $u\in S$, is only contained in the triangle $y_2uz_2$. 
If not, suppose $y_2uw$ is a new triangle distinct from  $y_2uz_2$. Since  $S$ is independent, we have $w\in \{x_1,y_1,z_1\}$. Let $u'\in S-\{u\}$, then $y_1z_1x_1,~wuy_2,~y_2u'z_2$ forms a  copy of $P_3^3$, a contradiction. Therefore, if we delete the two vertices $y_2,~z_2$ we destroy at most $|S|+9$ triangles.
By the condition $t(u)+t(v)\ge n-2$, we have $|S|\ge n-11$. We obtain that the total number of triangles is at most 
\[(n-11)\binom{11}{2}+\binom{11}{3}< \left\lfloor\frac{(n-1)^2}{4}\right\rfloor,\]
where the equality holds because $n\ge 300$. It follows that $G$ is not the extremal graph, and we are done.$\hfill\blacksquare$

\section{Concluding remarks} \label{concl}

In this paper, we studied the generalized Tur\'an number of edge blow-ups of the graphs $C_3^3$ and $P_3^3$.
It would be interesting to consider the general case of $C_k^3$ and $P_k^3$.
In this section, we pose two conjectures about the generalized extremal numbers of these graphs.

Let $H(n,p, t)$ denote the graph $K_{t-1}+T_p(n-t+1)$, where $T_p(n-t+1)$ is the balanced $p$-partite complete graph on $n-t+1$ vertices, i.e., the Tur\'an graph.  Let $H^+(n,p,t)$ be the graph obtained from $H(n,p,t)$ by adding an extra edge in any class of  $T_p(n-t+1)$.

Based on the our results and the Tur\'an number of $\ex(n,C_k^3)$ and $\ex(n,P_k^3)$, we pose the following conjecture for the generalized Tur\'an number.
 
 \begin{conjecture}
When $k\ge 4$ and $n$ is sufficiently large, $H(n,2,\lfloor\frac{k-1}{2}\rfloor+1)$ is the unique extremal graph for both $\ex(n,K_3,C_k^3)$ and $\ex(n,K_3,P_k^3)$ when $k$ is odd, and $H^+(n,2,\lfloor\frac{k-1}{2}\rfloor+1)$ is the unique extremal graph when $k$ is even.
 \end{conjecture}

\section{Acknowledgements}
The research of the authors Gy\H{o}ri and Salia was partially supported by the National Research, Development and Innovation Office NKFIH, grants  K132696 and SNN-135643. 
The research of Salia was supported by the Institute for Basic Science (IBS-R029-C4).
The research of Tompkins was supported by NKFIH grant K135800.

\bibliography{References.bib}

\end{document}